\documentclass[a4paper,11pt]{amsart}

\providecommand{\MR}{\relax\ifhmode\unskip\space\fi MR }

\providecommand{\href}[2]{#2}

\usepackage{enumerate, booktabs, mathrsfs}
\usepackage{amsmath, amsfonts, amssymb, amsthm}
\usepackage[all]{xy}
\usepackage[height=22.5cm, width=15.5cm]{geometry}
\usepackage[active]{srcltx}
\usepackage{color}  
\numberwithin{equation}{section}
\newtheorem{thm}{Theorem}[section]
\newtheorem{lemma}[thm]{Lemma}
\newtheorem{proposition}[thm]{Proposition}
\newtheorem{cor}[thm]{Corollary}
\theoremstyle{definition}
\newtheorem{definition}[thm]{Definition}
\newtheorem{remark}[thm]{Remark}

\numberwithin{equation}{section}

\author{Bruce Gilligan}
\address{Dept. of Mathematics \& Statistics, University of Regina, Regina, Canada S4S 0A2}
\email{gilligan@math.uregina.com}

\date{\today}   

\begin{document}

\title[Levi's problem for complex homogeneous manifolds]
{Levi's problem for complex homogeneous manifolds}

   \dedicatory{To the memory of R. Remmert}    

\begin{abstract}    
Suppose $G$ is a connected complex Lie group and $H$ is a closed complex subgroup.  
Then there exists a closed complex subgroup $J$ of $G$ containing $H$ such that 
the fibration $\pi:G/H \to G/J$ is the holomorphic reduction of $G/H$, i.e., $G/J$ is holomorphically 
separable and ${\mathcal O}(G/H) \cong \pi^*{\mathcal O}(G/J)$.  
In this paper we prove that if $G/H$ is pseudoconvex, i.e.,    if      
$G/H$ admits a continuous plurisubharmonic exhaustion function,           
then $G/J$ is Stein and $J/H$ has no non--constant holomorphic functions.   
\end{abstract}

\thanks{This work was partially supported by an NSERC Discovery Grant.  
We thank K. Oeljeklaus for pointing out Proposition \ref{Karl}.        
We also thank A. Huckleberry for stimulating discussions concerning this paper, 
C. Miebach for his constructive criticisms, and the referee for some further suggestions. }

\maketitle

\section{Introduction}   \label{sect1}  

The original Levi problem dealt with the characterization of domains of holomorphy in $\mathbb C^n$ having smooth boundary  
in terms of conditions on that boundary.  
Grauert \cite{Gr58} (resp. Narasimhan \cite{Nar62}) showed that a complex manifold (resp. complex space)   
that admits a smooth (resp. continuous) strictly plurisubharmonic exhaustion function is Stein.  
However, Grauert also pointed out that there exist pseudoconvex domains in tori all of whose holomorphic functions are constant, e.g., see \cite{Nar63}.  
This shows that the Levi problem for pseudoconvex domains that are not strongly pseudoconvex fails, in general.      

\medskip 
In this paper we restrict our attention to complex homogeneous manifolds $X := G/H$ with $G$ a connected complex Lie group and $H$ a closed complex subgroup,      
and we seek conditions under which its holomorphic function algebra ${\mathcal O}(G/H)$ is Stein.   
Setting 
\[    
       J \; := \;  \{ \ g \in G \   | \  f(gH) \;  = \;  f(eH) \mbox{  for all  } f\in{\mathcal O}(G/H) \  \}   
\]           
yields a closed complex subgroup of $G$ containing $H$ and consequently one has   
the holomorphic fibration $\pi : G/H \to G/J, \; gH \mapsto gJ$   
which we call the \textbf{holomorphic reduction} of $G/H$   \cite{GH78}.  
By construction  one has ${\mathcal O}(G/H) \cong \pi^* {\mathcal O}(G/J)$.   
The question then reduces to considerations of when the holomorphically separable complex homogeneous manifold $G/J$ is Stein.        

\medskip  
Now it is known that holomorphic separability implies Steinness for complex Lie groups \cite{MM60}, 
for complex nilmanifolds \cite{GH78}, and for complex solvmanifolds \cite{HO86}.  
But the situation is much different when the group acting is semisimple or reductive.  
For example, $\mathbb C^n \setminus \{ 0 \}$ is not Stein whenever $n>1$.  
What is known is that $G/H$ is holomorphically separable for $G$ reductive implies $H$ is an 
algebraic subgroup of $G$ \cite{BO73} and such a $G/H$ is Stein if and only if $H$ is reductive \cite{Mat60} and \cite{On60}.  
A reductive complex Lie group is the complexification of a totally real maximal compact subgroup and the  
compactness of this subgroup is playing an essential role.  
An excellent survey of the function theory on $G/H$ for $G$ reductive can be found in \cite[Chapter 5]{Akh95} and we refer   
the interested reader to that book and the references listed therein.

\medskip   
The present paper analyzes the holomorphic function theory of \textbf{pseudoconvex} complex homogeneous manifolds,   
where a complex manifold is pseudoconvex if it admits a continuous plurisubharmonic exhaustion function.       
In \cite[Main Theorem]{GMO13} we proved that the base of the holomorphic reduction of any pseudoconvex complex   
homogeneous manifold $G/H$ is Stein and its fiber has no non--constant holomorphic functions if $G$ is solvable  (see also  \cite[Theorem 8.5]{GMO13}) 
or if $G$ is reductive (see also \cite[Theorem 7.5]{GMO13}).         
Here we prove that this result holds more generally for pseudoconvex homogeneous manifolds of mixed groups, where by a mixed group       
we mean that $G$ is a connected, simply connected, complex Lie group that 
has a Levi--Malcev decomposition $G = S \ltimes R$, where $R$ is the radical of $G$,    
$S$ is a maximal semisimple subgroup, and both $\dim R > 0$ and $\dim S > 0$.  

\begin{thm} \label{mnthm}  
Let $G$ be a connected complex Lie group and $H$ a closed complex subgroup of $G$ such 
that $G/H$ is pseudoconvex.  
Suppose $G/H \to G/J$ is the holomorphic reduction of $G/H$.  
Then $G/J$ is Stein and ${\mathcal O}(J/H) = \mathbb C$.  
\end{thm}   

This note is organized as follows.   
Section two contains some technical results that are needed for the proof.  
The third section presents the proof of Theorem \ref{mnthm}.

 
\section{Technical Preparations}  

\subsection{The Hirschowitz annihilator}   \label{Hirschowitz}  
Every element $\xi\in\mathfrak g$, the Lie algebra of $G$, can be thought of as a right invariant vector field on $G$ 
and, as such, pushes down to a holomorphic vector field $\xi_X$ on any complex homogeneous manifold $X:=G/H$ of the group $G$.     
An inner integral curve in such a homogeneous space $X$ is a non--constant holomorphic map $\mathbb C \to X$ with relatively compact image 
in $X$ that is the integral curve of some vector field $\xi_X$ associated to some $\xi\in\mathfrak g$.  
Hirschowitz \cite{Hir75} considered such concepts in the context of infinitesimally homogeneous manifolds, where a manifold  
is infinitesimally homogeneous if every tangent space is generated by global holomorphic vector fields.  
A complex manifold that is homogeneous under the action of a Lie group of holomorphic transformations is 
infinitesimally homogeneous.     

\medskip   
Hirschowitz showed that a pseudoconvex, infinitesimally homogeneous $X$ that does not contain    
any inner integral curve is Stein \cite[Proposition 3.4]{Hir75}.   
This is the starting point of our investigations of pseudoconvex homogeneous manifolds that are not Stein given in \cite{GMO13}.      
By the maximum principle any plurisubharmonic function on a complex manifold $X$ is constant along every inner integral curve in $X$.   
One has to determine the ``directions of degeneracy'' of plurisubharmonic functions in terms of a certain subset of 
$\mathfrak g$ whose corresponding holomorphic vector fields ``kill them''.  
So in \cite{GMO13} we define the Hirschowitz annihilator $\mathcal A$ to be the connected Lie subgroup of $G$ whose Lie algebra is given by     
\[   
     \mathfrak a \; := \; \{ \ \xi\in\mathfrak g \ | \ \xi_X \varphi (x_0) \; = \; 0, \;  \forall   \varphi\in{\mathscr P(X)}  \ \}  ,  
\]  
where $\mathscr P(X)$ is the space of continuous plurisubharmonic functions on $X := G/H$.   
For continuous functions the derivative is understood in the sense of distributions.   
By definition $\mathfrak a$ is a complex vector subspace of $\mathfrak g$ and   
it is also a Lie subalgebra of $\mathfrak g$ that properly contains $\mathfrak h$, if $X$ is not Stein \cite[Lemma 3.3]{GMO13}.       
The corresponding (not necessarily closed) connected complex subgroup $\mathcal A$ of $G$ contains the identity component 
$H^0$ of $H$ and defines a complex foliation $\mathscr{F} = \{ F_x\}_{x\in X}$, the {\bf Levi foliation} of the manifold $X$.   
Every leaf $F_x$ of this foliation is a relatively compact immersed complex submanifold of $X$ containing every 
inner integral curve in $X$ passing through the point $x$ and is an orbit of the group $\mathcal A$.    
In general, complex foliations are rather difficult to understand.      
But here the foliation arises from a subgroup strongly reflecting the underlying geometry of the homogeneous manifold   
and is related to the existence of a plurisubharmonic exhaustion on it.  
This allows a sufficiently good understanding in 
order to analyze the structure from the point of view of its holomorphic function algebra.                 

\medskip   
If the leaves of the foliation are closed, then they are compact and $X$ is holomorphically convex \cite[Theorem 4.1]{GMO13}.  
Indeed, the holomorphic reduction is then given by the Remmert reduction \cite{Rem56} and has compact fiber and Stein base.   
The main difficulty occurs when this is not the case.    
The following observation is essential in what follows.      

\begin{remark} 
A question of Serre \cite{Ser53} asks whether the total space of every holomorphic fibration with 
Stein fiber and Stein base is itself Stein.  
Counterexamples are known, some of which are even homogeneous, e.g., \cite{CL85}.  
Among other things Lemma \ref{Hirsch} below asserts that Serre's question has an affirmative answer in the present setting.        
\end{remark}

\begin{lemma} [\cite{Hir75} and Remark 2.5 in \cite{GMO13}]  \label{Hirsch}   
Suppose $Y$ is a pseudoconvex complex homogeneous manifold that admits a holomorphic fibration  
\[  
              Y \; \stackrel{F}{\longrightarrow} \; B  
\]
with both $F$ and $B$ being holomorphically separable manifolds.   
Then $Y$ is Stein.   
\end{lemma}  

\begin{proof}  
Note that every complex homogeneous manifold is infinitesimally homogeneous.  
As neither $F$ nor $B$ can contain any inner integral curve,  
the same is true of $Y$.   
Therefore, $Y$ is Stein \cite[Proposition 3.4]{Hir75}.    
\end{proof}

\subsection{Tits' Normalizer Fibration}   

Suppose $H$ is a $k$--dimensional closed subgroup of an $n$--dimensional Lie group $G$.  
The Lie algebra $\mathfrak h$ of $H$ is a Lie subalgebra of $\mathfrak g$ and can be considered   
as a point in the Grassman manifold $Gr(k,n)$ of $k$--dimensional vector subspaces of $\mathfrak g$.  
Since ${\rm ad}(G) \subset GL(\mathfrak g)$, there is a natural action of ${\rm ad}(G)$ on $Gr(k,n)$ and 
the ${\rm ad}(G)$--orbit of the point $\mathfrak h$ can be identified with $G/N$, where $N := N_G(H^0)$, 
i.e., the normalizer in $G$ of the identity component $H^0$ of $H$.  
Via the Pl\"ucker embedding $Gr(k,n)$ can be realized as a submanifold of some projective space such 
that the automorphisms of the Grassman are restrictions of those automorphisms of the projective space 
that stabilize the embedded Grassman.  
In this way we realize $G$ acting linearly on $G/N$ via the adjoint representation.  
Since $H \subset N_G(H^0)$, we have the Tits normalizer fibration $G/H \to G/N$  \cite{Tits62}.      

\begin{remark}  
Let $I$ be a complex Lie subgroup of the complex Lie group $G$.   
The normalizer $N_G(I^0)$, where $I^0$ denotes the connected component 
of the identity of $I$  is a closed subgroup of $G$, since it is the 
isotropy subgroup of a point in a complex projective space under the adjoint action of the group $G$.  
\end{remark}

\subsection{Closure of orbits}  

Let $G$ be a (real) Lie group and $H$ a closed subgroup of $G$.   
For $I$ a normal Lie subgroup of $G$, set $F := cl_X(I.x_0)$,   
where $x_0$ is  the base point of $X := G/H$, and let $J := cl_G(I\cdot H)$. 

\begin{lemma}  \label{closure}   
Then $J$ is a subgroup of $Stab_G(F)$.  
\end{lemma} 

\begin{proof}  
Since $F$ is closed, $Stab_G(F)$ is closed.  
Clearly $I\cdot H \subset Stab_G(F)$.  
So $J \subset Stab_G(F)$.  
\end{proof}  

\begin{remark}  \label{closurermk}   
As a consequence, $F = J . x_0$.  
Now for $x\in X$ define $[x] := cl_X(I.x)$ and note that if $x=g(x_0)$, then $[x] = g(F)$.  
Thus the classes $[x]$ are the fibers of the homogeneous fibration $G/H \to G/J$.    
\end{remark}

\subsection{Fibrations over Projective orbits}  

We use the following notation for the derived series of the Lie group $G$ 
\[  
     G^{(0)} := G, \; G^{(1)} := G' = [G,G], \; \ldots, \; G^{(k)} := [G^{(k-1)},G^{(k-1)}] \mbox{ for all } k >1   .   
\]

\begin{lemma}  \label{flag1}  
Let $X := G/H$ be an orbit of a connected complex Lie group
$G$ acting holomorphically and effectively on some projective space.   
Assume $J$ is a connected, complex subgroup of $G$   
that has positive dimensional orbits in $G/H$ that are relatively compact.  
Then $J$ is a subgroup of every subgroup $G^{(m)}$ of the derived series of $G$.      
\end{lemma}   

\begin{proof} 
By a result of Chevalley \cite{Chev51} the image of the commutator group $G'$   
in the automorphism group of the ambient projective space is algebraically closed.  
This means that $G'$ is acting as an algebraic group on the projective space and,  
in particular, that its orbits in $G/H$ are closed.  
Thus one has the commutator fibration $G/H \to G/H\cdot G'$.  
Let $x_0\in G/H$ be the base point.  
The Stein Abelian group $\overline{G}/\overline{G}_{x_0}\cdot \overline{G}'$ contains $G/H\cdot G'$ as a $G$--orbit \cite{HO81},  
where the bar denotes the Zariski closure.    
Therefore, $J.x_0$ is contained in the fiber of $G/H \to G/H\cdot G'$ by the maximum principle with   
$J.x_0$ still being relatively compact in the closed fiber $G'/H\cap G'$.  
Replace $G$ (resp. $H$) by $G^{(1)} :=G'$ (resp. $H^{(1)}:=H\cap G' = G'_{x_0}$) and   
iterate the argument to see that $J$ is a subgroup of every group in the derived series of $G$.  
\end{proof}  

\begin{lemma}  \label{flag3}  
Let $G$ be a connected complex Lie group, $H$ a closed complex subgroup of $G$, and $I$ a closed complex subgroup of $G$ containing    
$H$ with $G/I$ equivariantly embedded in the complex projective space $\mathbb P_N$.     
Suppose that the fibration $G/H \to G/I$ is a covering.           
If $J$ is a connected, normal, complex Lie subgroup of $G$ whose orbits in $G/H$ are relatively compact, then the $J$--orbits in $G/H$ are flag manifolds.   
\end{lemma}  

\begin{proof}  
By Lemma \ref{flag1} the image of $J$ in the automorphism group of $G/I$   
lies in the image of every subgroup $G^{(m)}$ of the derived series of $G$.    
Since $G$ has finite dimension, one has $G^{(k)} = (G^{(k)})' = G^{(k+1)} = \ldots $ for some $k$.  
As the $J$--orbits have positive dimension, $G^{(k)}$ is a positive dimensional perfect Lie group that is acting algebraically on $\mathbb P_N$.     
Its radical $R_{G^{(k)}}$ is nilpotent  \cite{Jac62}.    
Since $J$ is normal, its radical $R_{J} = R_{G^{(k)}} \cap J$ \cite{Jac62}.     
Thus $R_{J}$ is a connected complex subgroup of the unipotent group $R_{G^{(k)}}$ and so 
is acting algebraically on $\mathbb P_N$ with each of its orbits a closed copy of $\mathbb C^q$ for some $q \ge 0$.   
Since the $J$--orbits are relatively compact, we must have $q = 0$.  
This implies that $R_J$ acts trivially and thus $J$ is acting algebraically as a semisimple group on $G/I$.

\medskip 
Now $A := {\rm cl}_{G}(J \cdot H) \subset K := {\rm cl}_{G}(J\cdot I)$, since $H \subset I$ by assumption.    
By Remark \ref{closurermk} and the assumption that the $J$--orbits are relatively compact we get the following diagram 
\[ 
        \begin{array}{ccc}   G/H & \longrightarrow & G/A \\ 
        \downarrow & & \downarrow \\ 
        G/I & \longrightarrow & G/K \end{array} 
\]  
The $J$--orbit through the base point in $G/I$ is contained in $K/I$ and the latter space is compact.   
Thus the closure of this $J$--orbit lies in $K/I$.  
Since $J$ is acting algebraically as a semisimple complex group, this orbit is Zariski open in its closure and     
its boundary consists of $J$--orbits of strictly lower dimension.  
Since $J \lhd G$, all orbits have the same dimension and so the boundary is empty, i.e., the $J$--orbits are closed and thus compact.  
Compact orbits of a complex Lie group acting holomorphically on a projective space are flag manifolds.  
Since flag manifolds are simply connected and the $J$--orbits in $G/H$ cover the $J$--orbits in $G/I$, it follows that 
the $J$--orbits in $G/H$ are flag manifolds.  
\end{proof}  

\begin{cor}  \label{flag2}  
Let $G$ be a connected complex Lie group, $H$ a closed complex subgroup of $G$, and $I$ a closed complex subgroup of $G$ containing    
$H$ with $G/I$ equivariantly embedded in the complex projective space $\mathbb P_N$   
and the fibers of the fibration $G/H \to G/I$ holomorphically separable.        
If $J$ is a connected, normal, complex Lie subgroup of $G$ whose orbits in $G/H$ are relatively compact, then the $J$--orbits in $G/H$ are flag manifolds.   
\end{cor}  

\begin{proof} 
Since $I/H$ is holomorphically separable, the $J$--orbits intersect the fibers of the fibration $G/H \to G/I$ transversally and 
so cover the corresponding $J$--orbits in $G/I$ which are necessarily relatively compact in $G/I$.       
The result now follows from Lemma \ref{flag3}. 
\end{proof}   

\subsection{Normality under closure and complexification}  
For $J$ a connected Lie subgroup of $G$ the complexification $J^{\mathbb C}$ of $J$ is the connected Lie subgroup of $G$ 
corresponding to the Lie algebra $\mathfrak j + i \mathfrak j$, where $\mathfrak j$ denotes the Lie algebra of $J$.        
Subsequently we need to consider what happens to normality under closure and complexification and so the following is important.   

\begin{lemma}  \label{normal}  
Suppose $I$ is a connected normal Lie subgroup of a connected Lie subgroup $J$ of a connected complex Lie group $G$.  
Then 
\begin{enumerate} 
\item $I \lhd {\rm cl}_G(J)$  
\item $I^{\mathbb C} \lhd J^{\mathbb C}$
\end{enumerate}  
In particular, $I^{\mathbb C} \lhd \widetilde{J}$, where $\widetilde{J}$ is the smallest connected closed complex subgroup of $G$ that contains $J$.  
\end{lemma}   

\begin{proof} 
(1) If $\lim g_n = g \in {\rm cl}_G(J)$,where $g_n\in J$,  then $gIg^{-1} = \lim g_n I g_n^{-1} \subset I$.  

\noindent  
(2) This follows from the fact that 
$ [\mathfrak i, \mathfrak j ] \subset \mathfrak i \Longrightarrow [ \mathfrak i + i \mathfrak i , \mathfrak j + i \mathfrak j ] \subset \mathfrak i + i \mathfrak i$.  

\noindent  
Note that $\widetilde{J}$ can be formed by alternately taking the complexification and closure of $J$.  
Applying (1) and (2), as appropriate, completes the proof.  
\end{proof}

\subsection{Existence of a fibration by a solvmanifold}   

\begin{proposition}[personal communication from K. Oeljeklaus] \label{Karl}    
Let $X := G/\Gamma$, where $\Gamma$ is a discrete subgroup of a connected, simply connected,   
complex Lie group $G$ with Levi decomposition $G = S \ltimes R$ and $\dim R > 0$.       
Then there is a connected, complex, solvable subgroup $H$ of $G$ normalized by
$\Gamma$, containing $R$, with $H\cdot \Gamma$ a closed subgroup of $G$.   
In particular, one has the proper fibration   (unless $S = \{ e \}$)   
\begin{eqnarray} \label{radfibr}  
       G/\Gamma \;  \longrightarrow \;  G/H\cdot \Gamma \; = \; S/S\cap H\cdot \Gamma   
\end{eqnarray}  
that has the connected complex solvmanifold $H/H \cap \Gamma$ as its typical fiber.    
\end{proposition}

\begin{proof}  
If the $R$--orbits themselves are closed, set $H := R$.    
If not, then the Zassenhaus Lemma \cite{Aus63} is used in \cite[Theorem 2]{Gi81} to show the existence of    
a minimal connected complex solvable subgroup $H_1 \subset G$ normalized by $\Gamma$ 
and containing the identity component of ${\rm cl}_G(R\cdot\Gamma)$.   
If $H_1\cdot\Gamma$ is closed in $G$, one has the desired result with $H := H_1$.    

\medskip 
Otherwise, let $N_1 := N_G(H_1)$. 
Since $\Gamma$ normalizes the identity component $N_1^0$ of $N_1$, it also normalizes its radical $R_1$.    
Now $H_1 \subset R_1$, because $H_1$ is solvable and normal in $N_1$.    
Either $R_1\cdot\Gamma$ is closed in $G$, or the identity component of ${\rm cl}_G(R_1\cdot\Gamma)$ 
is contained in $N_1^0$.    
Applying the Zassenhaus Lemma again (in $N_1^0$) we see that this identity component is solvable 
and normalized by $\Gamma$,  since $N_1^0\cdot \Gamma  = N_1$ is closed in $G$.    
Let $H_2$ be the smallest connected closed complex subgroup of $G$ that contains 
this identity component.   
Then $H_2$ is solvable, normalized by $\Gamma$ and its dimension is strictly greater than the dimension of $H_1$.   
A finite number of steps yields the desired connected complex solvable group $H$.     
\end{proof}   

\subsection{Existence of a tower}   

\begin{lemma}  \label{ind1}  
Suppose $\Gamma$ is a cocompact, discrete subgroup of a (positive dimensional) connected solvable Lie group $L$ such that  
$L^{\mathbb C}/\Gamma$ is Stein, where $L^{\mathbb C}$ is the complexification of $L$.  
Then there exists a fibration by the center $Z$ of the nilradical of $L^{\mathbb C}$   
\[   
          L^{\mathbb C}/\Gamma  \; \stackrel{(\mathbb C^*)^k}{\longrightarrow} \; L^{\mathbb C}/Z\cdot\Gamma 
\]    
that has $(\mathbb C^*)^k$ as fiber with $k>0$.      
\end{lemma}   

\begin{proof} 
Let $N_0$ be the nilradical of $L$ and $N$ the nilradical of $L^{\mathbb C}$.   
Then $N_0$ has closed orbits in $L/\Gamma$ by a theorem of Mostow \cite{Mos54} or \cite{Mos71}    
and thus $N$ has closed orbits in $L^{\mathbb C}/\Gamma$.     
Let $Z$ be the center of $N$ (resp. $Z_0$ of $N_0$).   
Incidentally, note that $\dim Z > 0$ \cite{Mat51}.  
Since $N/N\cap\Gamma$ is a closed complex submanifold of the Stein manifold $L^{\mathbb C}/\Gamma$,   
we see that $N/N\cap\Gamma$ is Stein.      
Hence the subgroup $Z\cdot\Gamma$ is closed by a result of Barth--Otte \cite{BO69}; see also \cite[Theorem 4]{GH78}.    
Therefore, we have the fibration 
\[   
        L^{\mathbb C}/\Gamma \; \longrightarrow \; L^{\mathbb C}/Z\cdot\Gamma  .   
\]    
It follows that the fibers of the fibration above are $(\mathbb C^*)^k$--orbits for some positive integer $k$; see \cite[Theorem 7]{GH78}.          
\end{proof}  

\begin{definition}  
A $\mathbb C^*$ power tower of length {\it one} is simply the manifold $(\mathbb C^*)^p$ for some positive integer $p$.   
For any integer $n>1$ a $\mathbb C^*$ power tower of length $n$ is a $(\mathbb C^*)^k$--bundle 
over a $\mathbb C^*$ power tower of length $n-1$.      
\end{definition}   

\begin{remark} 
Repeated application of Lemma \ref{ind1} shows that the space $L^{\mathbb C}/\Gamma$ is a $\mathbb C^*$ power tower of length $n$ 
for some positive integer $n$.  
\end{remark}


\section{Proof of the main result}   

\subsection{Formulation of the strategy of the proof}  
\begin{remark} \label{connected}  
There is a technical point that can arise in our construction,      
in that an intermediary fibration whose fiber is not connected might be involved.      
This is handled by a type of {\em Stein factorization for homogeneous fibrations}.      
Suppose $G$ is a connected Lie group that contains a closed subgroup $I$ containing a closed subgroup $H$.  
Let $\widetilde{I}$ be those connected components of $I$ that meet $H$.  
Then $\widetilde{I}$ is a closed subgroup of $G$ containing $H$ and  the fibration $G/H \to G/\widetilde{I}$ has connected fiber $\widetilde{I}/H$.    
\end{remark}

\noindent 
{\sc Strategy of the Proof:}

\medskip\noindent   
Assume $G/H$ is a pseudoconvex homogeneous manifold that is not Stein.  
In order to prove the theorem 
we construct a closed complex subgroup $I$ of $G$ containing $H$ with $\dim I > \dim H$ and $I/H$ connected,   
possibly with the aid of Remark \ref{connected}, such that    
\begin{enumerate}  
\item [] (i)  ${\mathcal O}(I/H) = \mathbb C$  and    
\item [] (ii) every continuous plurisubharmonic  
exhaustion function on $G/H$ induces a continuous plurisubharmonic exhaustion function on $G/I$.   
\end{enumerate}  
Then if $G/H \to G/J$ is the holomorphic reduction, $I$ is a subgroup of $J$ because of (i), $G/I$ is pseudoconvex because of (ii) 
and is either Stein (then $I=J$ and we are done) or not Stein and one applies the construction until 
one does reach the holomorphic reduction; see also the last paragraph of the proof below for more details.      

\begin{remark} 
Here is a list of some complex homogeneous manifolds that do satisfy (ii), 
whenever they occur as the fiber of a homogeneous fibration $G/H \to G/I$.  
For (2) and (3), the basic tool to prove this is Kiselman's Minimum Principle \cite[Theorem 2.2]{Kis78}:  
\begin{enumerate}  
\item  compact complex homogeneous manifolds,       
\item  Cousin groups, see \cite[Lemma 6.1 (1)]{GMO13},   
\item  the fibers of certain $\mathbb C^*$ power towers provided the exhaustion function is 
constant on the underlying circle power tower, see below.        
\end{enumerate} 
\end{remark}   

\subsection{The proof itself}    

\begin{proof}  
Let $G/H$ be a non--Stein pseudoconvex homogeneous manifold.      
Then its  Hirschowitz annihilator  $\mathcal A$ satisfies $\dim \mathcal A > \dim H$.    
We define subgroups by setting      
\begin{enumerate} 
\item $G_{1} := N_{G}(\mathcal A)$; note that $H$ is a subgroup of $N_G(\mathcal A)$ by \cite[p. 42]{GMO13} and   
\item $G_{2} := N_{G_{1}}(H^0)$.  
\end{enumerate}  
Note that in (1), because the $\mathcal A$--orbits are positive dimensional and relatively compact, the fibration 
$G/H \to G/N_{G}(\mathcal A)$ is not a covering and its fiber is not Stein.   
Also because of the fact that  $\mathcal A$ is normal in $G_1$, if the fibration $G_1/H \to G_1/N_{G_{1}}(H^0)$ 
is a covering (resp. has a Stein fiber), then the $\mathcal A$--orbits in $G_1/H$ are flag manifolds by Lemma \ref{flag3} 
(resp. by Corollary \ref{flag2}).  
We are done, since setting $I := \mathcal A \cdot H$ yields a fibration $G/H \to G/I$ with $I/H$ compact and thus satisfying (i) and (ii).  
Thus we need only consider the setting where $G_2/H$ is a non--Stein pseudoconvex homogeneous manifold.    
If the Hirschowitz annihilator $\mathcal A_2$ for $G_2/H$ is not normal in $G_2$,   
we apply (1) again setting $G_3 := N_{G_{2}}(\mathcal A_2)$.  
So $G_3 = N_{G_3}(\mathcal A_2) < N_{G_1}(H^0)$, the latter because $G_3$ is a subgroup of $G_2$.       
Now set $\widehat{G} := G_3/H^0, \widehat{\mathcal A} := \mathcal A_2/H^0$ and $\Gamma := H/H^0$ and   
note that $\widehat{\mathcal A} \lhd \widehat{G}$ and has positive dimensional orbits in $\widehat{G}/\Gamma$.   
Furthermore, we may assume that $\widehat{G}$ is not solvable (resp. semisimple) because of 
\cite[Theorem 8.5]{GMO13} (resp. \cite[Theorem 7.1]{GMO13}), since each of these settings directly 
yields the desired subgroup $I$ satisfying (i) and (ii); in the first case it arises from a fibration by a Cousin group and in the second by a 
compact complex manifold.      
Hence we reduce to the case where $\widehat{G}$ is a mixed group.   
The rest of the construction below produces a subgroup $\widehat{I}$ of $\widehat{G}$ containing $\Gamma$   
with $\widehat{I}/\Gamma$ satisfying (i) and (ii).  
Taking the preimage $I$ of $\widehat{I}$ via the quotient homomorphism  $G_3 \to G_3/H^0$ gives us the desired fibration $G/H \to G/I$.  
For notational convenience we suppress the hats from now on and write $G$ instead of $\widehat{G}$, etc.    
 
\medskip  
Since $\mathcal A \lhd G$, we may apply Lemma \ref{closure} and Remark \ref{closurermk}.       
Set $L := cl_{G}(\mathcal A\cdot \Gamma)$ and let $\widetilde{L}$ be the smallest connected, closed, complex subgroup of $G$ that contains $L$.           
Since the $\mathcal A$--orbits are relatively compact, $L/\Gamma$ is compact.    
As a consequence, ${\mathcal O}(\widetilde{L}/\Gamma) = \mathbb C$ by the maximum and identity principles.           
We claim that we may further reduce to the setting where the group $L$ has a positive dimensional radical $R_L$, a fact that we will later use.            
If $\widetilde{L}$ is semisimple, then one has the fibration $G/\Gamma \to G/\widetilde{L}$ and  
$\widetilde{L}/\Gamma$ is pseudoconvex and thus holomorphically convex \cite[Theorem 7.1]{GMO13}.  
This implies $\widetilde{L}/\Gamma$ is compact and we are again done with $I := \widetilde{L}$.                
In particular, in the rest of the proof we assume that $\dim R_L > 0$.    

\medskip   
We need to show that $\widetilde{L}/\Gamma$ satisfies (ii) in this setting.    
By Proposition \ref{Karl} there is a fibration   
\[   
       \widetilde{L}/\Gamma \; \longrightarrow \;  \widetilde{L}/ \widetilde{H}\cdot\Gamma     
\]    
with $ \widetilde{H}\cdot\Gamma$ closed in $ \widetilde{L}$, where $ \widetilde{H}$ is a connected, solvable, complex Lie group   
that contains the radical $R_{ \widetilde{L}}$ of $\widetilde{L}$ and is normalized by $\Gamma$.    
Now we claim that we may further reduce to the setting where  
the fiber $ \widetilde{H}\cdot\Gamma/\Gamma$ of the above fibration is Stein.  
Otherwise, $ \widetilde{H}\cdot\Gamma/\Gamma$ would be a connected, pseudoconvex solvmanifold that is not Stein  
and there would exist a closed complex subgroup $I$ of $ \widetilde{H}\cdot\Gamma$ containing $\Gamma$ with $I/\Gamma$  
a positive dimensional Cousin group \cite[Theorem 8.5]{GMO13}.   
Clearly, $I/\Gamma$ satisfies conditions (i) and (ii).     
So from here on we may assume that $ \widetilde{H}\cdot\Gamma/\Gamma$ is Stein.  

\medskip  
Now in order to finish the proof that $\widetilde{L}/\Gamma$ satisfies (ii) we have to analyze the structure of (at least part of)    
the intersection of the compact orbit $L/\Gamma$ with the Stein orbit $\widetilde{H}/\widetilde{H}\cap\Gamma$.      
Consider the connected, real, solvable group $B := (L \cap \widetilde{H})^0$.  
Since $R_L \subset R_{\widetilde{L}}\subset\widetilde{H}$ by Lemma \ref{normal} and $R_L \subset L$, it follows that 
$R_L \subset B$ and thus $\dim B >0$.   
Let $B^{\mathbb C}$ be its complexification.      
We have $\Gamma \subset N_{\widetilde{L}}(B^{\mathbb C})$,   
since $\Gamma \subset L$ and $\Gamma$ normalizes $\widetilde{H}$ by Proposition \ref{Karl},   
and thus we may consider the fibration $\widetilde{L}/\Gamma \to \widetilde{L}/N_{\widetilde{L}}(B^{\mathbb C})$.   
Now we may assume that $N_{\widetilde{L}}(B^{\mathbb C})/\Gamma$ is not Stein, 
for, otherwise, $\mathcal A$ would have compact orbits by Lemma \ref{flag2}, a case that could be easily handled, as above.      
The Hirschowitz annihilator for the space $N_{\widetilde{L}}(B^{\mathbb C})/\Gamma$ need not be normal in $\widetilde{L}$.     
So we begin the proof again with the space $N_{\widetilde{L}}(B^{\mathbb C})/\Gamma$  and run   
at most a finite number of times (because $\dim G/H <\infty$) through all of its steps until   
the only situation demanding further attention occurs when $N_{\widetilde{L}}(B^{\mathbb C})=\widetilde{L}$.  
Hence we may assume that $B^{\mathbb C}$ is a connected, normal, solvable subgroup of ${\widetilde{L}}$.  
As a consequence, $B^{\mathbb C} \cap L$ is normal in $L$ and this implies that  $B \subset R_L$.  
But $R_L \subset B$ and so $B=R_L$.      
We claim that the $R_L$--orbits in $\widetilde{L}/\Gamma$ are compact.    
As noted above, $R_L \subset R_{\widetilde{L}}$ by Lemma \ref{normal}.  
So by the construction of $\widetilde{H}$, one has 
\[   
          (R_L\cdot\Gamma)^0 \; \subset \; {\rm cl}_{\widetilde{L}}(R_L\cdot\Gamma)^0 \; \subset \; 
          ((L\cap\widetilde{H})\cdot\Gamma)^0 \; = \; (R_L\cdot\Gamma)^0 ,      
\]   
where, as usual, the superscript denotes the connected component of the identity.    
It follows that the $R_L$--orbits are closed and since the $L$--orbits are compact, the $R_L$--orbits are thus also compact.   
Note that if ${\mathfrak r}_L \cap i {\mathfrak r}_L \not= (0)$, then the connected complex Lie group corresponding to 
the complex ideal ${\mathfrak r}_L \cap i {\mathfrak r}_L $ has positive dimensional orbits in the compact $R_L$--orbits in the $\widetilde{H}$--orbits.  
By the maximum principle this contradicts our reduction to the setting where the $\widetilde{H}$--orbits are Stein.  
Thus one must have ${\mathfrak r}_L \cap i {\mathfrak r}_L = (0)$ and the $R_L$--orbits in $\widetilde{L}/\Gamma$ are totally real.  
Now consider the complexification $R_L^{\mathbb C}$ of $R_L$ that has 
Lie algebra ${\mathfrak r}_L^{\mathbb C} :=  {\mathfrak r}_L \oplus i {\mathfrak r}_L$.   
According to a conjecture of Mostow \cite{Mos54} that was proved by Auslander--Tolimieri \cite{AT69} and Mostow \cite{Mos71} every solvmanifold has the 
structure of a (real) vector bundle over a compact solvmanifold.   
In general, the compact base is homogeneous with respect to a group that is not a subgroup of the original solvable Lie group acting on the manifold,   
but rather lies in a certain algebraic hull of that group.    
We claim that our setting is special in that $R_L \subset R_L^{\mathbb C}$ can be taken to be that subgroup.      
Suppose 
\[  
         R_L^{\mathbb C}/R_L^{\mathbb C}\cap\Gamma \; \stackrel{\mathbb R^k}{\longrightarrow} \; M  
\]  
is the vector bundle given by Mostow's conjecture.  
Since $R_L^{\mathbb C}/R_L^{\mathbb C}\cap\Gamma$ is Stein, it follows from Serre's homology condition \cite{Ser53} that   
$ \dim_{\mathbb C} R_L^{\mathbb C} \ge  \dim_{\mathbb R} M \ge  \dim_{\mathbb R} R_L$.      
On the other hand  we have $\dim_{\mathbb R} R_L =  \dim_{\mathbb C} R_L^{\mathbb C}$, as noted above.      
As a consequence, the compact base $M$ is diffeomorphic to $R_L/R_L\cap\Gamma$ and $R_L^{\mathbb C}\cap\Gamma = R_L\cap\Gamma$.  
So the $R_L^{\mathbb C}$--orbits are closed in $\widetilde{L}/\Gamma$.  
We may now apply Lemma \ref{ind1} to the triple $(R_L \cap\Gamma, R_L, R_L^{\mathbb C})$.   

\medskip   
Now we have the fibration $\widetilde{L}/\Gamma \to \widetilde{L}/R_L^{\mathbb C}\cdot\Gamma$.    
Any continuous plurisubharmonic exhaustion function $\varphi$ on $G/\Gamma$ is constant on each of the $\mathcal A$--orbits,  
since these orbits lie in the level sets of the exhaustion function.       
By continuity $\varphi$ is then constant on the orbits of the closure $L$ and thus on the orbits of its radical $R_L$.  
The $(S^1)^k$--orbits that arise in fibration   
\[   
         R_ L^{\mathbb C}/(R_L^{\mathbb C}\cap\Gamma)  \;   
         \stackrel{(\mathbb C^*)^k}{\longrightarrow} \; R_L^{\mathbb C}/Z\cdot(R_L^{\mathbb C}\cap\Gamma)   
\]    
given by Lemma \ref{ind1} are part of the $R_L$--orbits.  
Thus $\varphi$ is constant on the $(S^1)^k$--orbits.                    
As in  \cite[Lemma 6.1 (2)]{GMO13} one can apply Kiselman's minimum principle \cite[Theorem 2.2]{Kis78} 
and push $\varphi$ down to $R_L^{\mathbb C}/Z\cdot(R_L^{\mathbb C}\cap\Gamma)$.    
As a consequence $ \widetilde{L}/R_L^{\mathbb C}\cdot\Gamma$ is pseudoconvex.  
We continue in this fashion until a maximal semisimple group is acting transitively on the resulting quotient space $Z$.  
But then $Z$ is pseudoconvex and thus holomorphically convex \cite[Theorem 7.1]{GMO13}.  
Since $\mathcal{O} (\widetilde{L}/\Gamma) = \mathbb C$, it follows that $Z$ is compact and 
we see that $\widetilde{L}/\Gamma$ satisfies (ii).  
This completes the proof that $\widetilde{L}/\Gamma$ satisfies (i) and (ii).       

\medskip  
In order to complete the proof of the Theorem, we assume that $G/H$ is a pseudoconvex homogeneous manifold 
that is not Stein with ${\mathcal O}(G/H) \not= \mathbb C$ and we let $G/H \to G/J$ be its holomorphic reduction.  
Note that $J/H$ cannot be Stein, since this would imply that $G/H$ itself would be Stein by Lemma \ref{Hirsch}.  
We now choose the \emph{maximal} $I$ given by the construction above.       
Then ${\mathcal O}(I/H) = \mathbb C$.  
Moreover, $G/I$ is pseudoconvex due to the fact that $I$ satisfies (ii).  
If $G/I$ were not Stein, then there would exist a subgroup $I_1$ with $\dim I_1 > \dim I$ 
with $I_1/I$ satisfying conditions (i) and (ii).  
But this would imply that $I_1/H$ also satisfies these two conditions, contradicting the maximality of $I$.      
Consequently, $I = J$ and shows that the holomorphic reduction has the desired properties.  
\end{proof}


\begin{thebibliography}{Prop}   

\bibitem{Akh95}  
   Akhiezer, D. N.: Lie group actions in complex analysis. Aspects of Mathematics, E27. Friedr. Vieweg \& Sohn, Braunschweig, 1995. viii+201 pp. 

\bibitem{Aus63}  Auslander, L.: \emph{On radicals of discrete subgroups of Lie groups}, 
   Amer. J. Math. \textbf{85}, 145 -- 150 (1963).    
   
\bibitem{AT69}    
   Auslander, L.,  R. Tolimieri, R.:  \emph{On a conjecture of G. D. Mostow and the structure
   solvmanifolds},  Bull. Amer. Math. Soc.  \textbf{75}, 1330 -- 1333 (1969).   
   
\bibitem{BO69} 
    Barth, W., Otte, M.: \emph{\"{U}ber fast-uniforme Untergruppen komplexer Liegruppen 
    und aufl\"{o}sbare komplexe Mannigfaltigkeiten.}   Comment. Math. Helv. \textbf{44},  269 -- 281 (1969). 

\bibitem{BO73}  
    Barth, W., Otte, M.:  \emph{Invariante holomorphe Funktionen auf reduktiven Liegruppen.}   
    Math. Ann. \textbf{201}, 97 -- 112 (1973).  
     
\bibitem{Chev51}       
    Chevalley, C.:      Th\'eorie des groupes de Lie. Tome II. Groupes alg\'ebriques.    
    Actualit\'es Sci. Ind. no. 1152. Hermann \& Cie., Paris, 1951.    
    
\bibitem{CL85}  
    Coeur\'e G., Loeb,  J.-J.:  \emph{A counterexample to the Serre problem with a bounded domain of $\mathbb C^2$ as fiber.}      
    Ann. of Math. (2) \textbf{122}, no. 2, 329 -- 334 (1985).    

\bibitem{Gi81}  
    Gilligan, B.:  \emph{Ends of complex homogeneous manifolds having nonconstant holomorphic functions.} 
    Arch. Math. (Basel) \textbf{37}, no. 6, 544 -- 555 (1981).  

\bibitem{GH78}
   Gilligan, B.,  Huckleberry, A. T.:  \emph{ On non-compact complex nil-manifolds.} Math. Ann. \textbf{238}, no. 1, 39 -- 49 (1978). 
   
\bibitem{GMO13} 
    Gilligan, B., Miebach, C., Oeljeklaus, K.:      
   \emph{Pseudoconvex domains spread over complex homogeneous manifolds,}  Manuscripta Math. \textbf{142}, 35 -- 59 (2013).   
   
\bibitem{Gr58}
    Grauert, H.: \emph{On Levi's problem and the imbedding of real--analytice manifolds.}  Ann. of Math. 
    \textbf{68} (1958), 460 -- 472.   
 
\bibitem{Hir75} 
   Hirschowitz, A.: \emph{Le probl\`{e}me de L\'{e}vi pour les espaces homog\`{e}nes}.  Bull. Soc. Math. France 
   \textbf{103}(2), 191 -- 201 (1975).     
  
\bibitem{HO81} 
  Huckleberry, A. T.,  Oeljeklaus, E.: \emph{Homogeneous spaces from a complex analytic view--point.} 
  In: Manifolds and Lie Groups (Notre Dame, Ind., 1980) Progr. Math., vol. 14, pp. 159 -- 186, Birkh\"{a}user, Boston, 1981. 
  Edited by J. Hano, A. Morimoto, S. Murakami, K. Okamoto and H. Ozeki.   
    
\bibitem{HO86}   
    Huckleberry, A. T.,  Oeljeklaus, E.:   \emph{On holomorphically separable complex solv-manifolds.}  
    Ann. Inst. Fourier (Grenoble) \textbf{36}, no. 3, 57 -- 65 (1986).
 
\bibitem{Jac62} 
   Jacobson, N.:  Lie algebras. Interscience Tracts in Pure and Applied Mathematics, 
   No. 10 Interscience Publishers (a division of John Wiley \& Sons), New York-London 1962.    
  
\bibitem{Kis78} 
   Kiselman, C. O.:   \emph{The partial Legendre transformation for plurisubharmonic functions}. 
   Invent. Math. \textbf{49},  137 -- 148 (1978).      
   
\bibitem{Mat51}  
   Matsushima, Y.: \emph{On the discrete subgroups and homogeneous spaces of nilpotent Lie groups.}  
    Nagoya Math. J. \textbf{2}, 95 -- 110 (1951). 
   
\bibitem{Mat60}  
   Matsushima, Y.:   \emph{Espaces homog\`{e}nes de Stein des groupes de Lie complexes.} Nagoya Math. J, \textbf{16} (1960),  205 -- 218.
   
\bibitem{MM60}  
    Matsushima, Y., Morimoto, A.: \emph{Sur certains espaces fibr\'{e}s holomorphes sur une vari\'{e}t\'{e} de Stein.}   
    Bull. Soc. Math. France \textbf{88} (1960), 137 -- 155.        
   
\bibitem{Mos54}  
     Mostow, G. D.:  \emph{Factor spaces of solvable groups.}  Ann. of Math. (2) \textbf{60}, 1 -- 27 (1954).   
     
\bibitem{Mos71}    
     Mostow, G. D.:  \emph{Some applications of representative functions to solvmanifolds}, 
     Amer. J. Math. \textbf{93}, 11 -- 32 (1971). 

\bibitem{Nar62} 
   Narasimhan, R.: \emph{The Levi problem for complex spaces.} I, II, Math. Ann. \textbf{142} (1961) 355 -- 365; ibid. 
   \textbf{146} (1962), 195 -- 216.   

\bibitem{Nar63} 
   Narasimhan, R.: The Levi problem in the theory of functions of several complex variables. 
   1963 Proc. Internat. Congr. Mathematicians (Stockholm, 1962) pp. 385 -- 388.   

\bibitem{OR88}   
     Oeljeklaus, K.,  Richthofer, W.:  \emph{On the structure of complex solvmanifolds.}  
     J. Differential Geom. \textbf{27}, 399 - 421 (1988).  
    
\bibitem{On60}  
    Onishchik, A.:  \emph{Complex hulls of compact homogeneous spaces.}  Dokl. Akad. Nauk SSSR \textbf{130} 726 -- 729 (Russian);   
    translated as Soviet Math. Dokl. \textbf{1} (1960),  88 -- 91.
    
\bibitem{Rem56}  
   Remmert, R.: \emph{Sur les espaces analytiques holomorphiquement s\'eparables et holomorphiquement convexes.}      
   C. R. Acad. Sci. Paris \textbf{243}, 118 -- 121 (1956).
    
\bibitem{Ser53}  
    Serre, J.-P.:  \emph{Quelques probl\`emes globaux relatifs aux vari\'et\'es de Stein.}    
    Colloque sur les fonctions de plusieurs variables, tenu \`a Bruxelles, 1953, pp. 57 -- 68. Georges Thone, Li\`ege; Masson \& Cie, Paris, 1953.     
    
\bibitem{Tits62} 
   Tits, J.:  \emph{Espaces homog\`enes complexes compacts.}   
   Comment. Math. Helv. \textbf{37}, 111 -- 120 (1962/1963).

\end{thebibliography}
\end{document}